\documentclass[twoside,11pt,reqno]{amsart}
\usepackage{macros}

\title{Testing local-global divisibility at a small set of primes}
\author{Alexander B. Ivanov}
\address{Alexander Ivanov, Ruhr-Universit\"at Bochum, Fakult\"at f\"ur Mathematik, IB 2/153, 
Universit\"atsstrasse 150, D-44780 Bochum}
\email{alexander.ivanov@ruhr-uni-bochum.de}
\author{Laura Paladino}
\address{Laura Paladino, Universit\`a della Calabria, Ponte Bucci, Cubo 30B, 87075, Rende (CS), Italy} 
\email{laura.paladino@unical.it}
\begin{document}

%\addbibresource{bib_ADLV}

 \begin{abstract}
We show that the local-global divisibility in commutative algebraic groups defined over number fields
can be tested on sets of primes of arbitrary small density, i.e.\ stable and persistent sets. 
We also give a new description of the cohomological group giving an obstruction to the problem. In addition, we show new
examples of stable sets.
 \end{abstract} \normalcolor

\maketitle

\section{Introduction}
Let $k$ be a number field. Let $A$ be a connected commutative group scheme of finite type over $k$. Dvornicich and Zannier investigated a local-global principle for divisibility of rational points on $A$:

\begin{problem}[Local-global divisibility problem \cite{DvornicichZ_01}] \label{problem:DZ}
Let $r$ be a positive integer and let $P \in A(k)$. Suppose that for all but finitely many primes $\mfp$ of $k$, we have $rD_\mfp = P$ for some $D_{\mfp} \in A(k_{\mfp})$, where $k_{\mfp}$ denotes the completion of $k$ at $\mfp$. Can one conclude that there exists $D \in A(k)$ with $rD = P$?
\end{problem}

Of course, one may assume that $r$ is a power of a rational prime without loosing any generality. The solution to Problem  \ref{problem:DZ} and to variants of it are known in many cases, in particular, for tori \cite{DvornicichZ_01, Illengo_08} and for elliptic curves  \cite{PaladinoRV_12, PaladinoRV_14, Cre}. In elliptic curves the answer is affirmative for every power $r=p^n$, with $p>B(d)$, where
 $d:=[k:\bQ]$ and 
\[
B(d):=\left\{\begin{array}{ll}
3 & if \hspace{0.1cm} k=\bQ;\\
 (3^{\frac{d}{2}}+1)^2 & if \hspace{0.1cm} d>1.\\
\end{array}
\right.
\]

\noindent 
Observe that for $k\neq \bQ$ the bound $B(d)=(3^{\frac{d}{2}}+1)^2$ is the one appearing in \cite{Oesterle_}, giving an effective version of Merel's Theorem on torsion points of an elliptic curve \cite{Merel_96}, and in particular $E(k)[p] = 0$ for all primes $p>B(d)$.
It is shown in \cite[\S5]{DvornicichP_21} that for a \emph{fixed} elliptic curve $E$, a \emph{fixed} number field $k$ and a \emph{fixed} power $r$, Problem \ref{problem:DZ} admits an explicit and effective solution, that is there is an effectively computable constant $C(k,E,r)>0$ such that to deduce global divisibility, it suffices to test the local one for all primes $\fp$ of $k$ with norm $N\fp<C(k,E,r)$. 

\smallskip

In this note we ask, whether in Problem \ref{problem:DZ} one can considerably shrink the set of primes where local divisibility is tested, simultaneously for \emph{all} $A$,  \emph{all} $k$ and  \emph{all} $r$. It turns out that this strengthened version of the problem admits a solution. More precisely, we show that certain sets of primes with arbitrary small Dirichlet density suffice, cf.\ Theorem \ref{thm:p_stable_sets_test_p_divisibility}.
For clarity, let us state our main result in the case of elliptic curves, where the original Problem \ref{problem:DZ} is quite well-understood.

\begin{thm}[see Corollary \ref{cor:for_absolutely_stable_S} for most general statement]\label{thm:small_test_set} 
For any $\varepsilon > 0$ there exists a set $S$ of primes of $\bQ$, such that for all number fields $k$, all elliptic curves $E/k$ and all primes $p$ one has $0 < \delta_k(S) < \varepsilon$ and the following hold:
\begin{itemize}
\item[(1)] If a point $P \in E(k)$ is locally divisible by $p$ at any $\fp \in S_k$, then it is globally divisible by $p$ in $E(k)$.
\item[(2)] Suppose that $p > B(d)$. Then for all $n \geq 1$, if a point $P \in E(k)$ is locally divisible by $p^n$ at any $\fp \in S$, then it is globally divisible by $p^n$  in $E(k)$.
\end{itemize}
\end{thm}

Moreover, sets $S$ as in Theorem \ref{thm:small_test_set} exist in abundance, and examples of them can be given explicitly, cf.\ \S\ref{sec:stable_set}. 

It is well known that an obstruction to the validity of Problem \ref{problem:DZ} is given by
the first local cohomology group (see \cite{DvornicichZ_01},\cite{DvornicichZ_07} and Equation \eqref{eq:cohomology_with_local_conditions} below).
Such a group is isomorphic to some modified Tate-Shafarevich groups defined by sets $S$ of primes of density one (see for instance \cite{Creutz_12,DP_22}). If one shrinks the set $S$, the Tate-Shafarevich group can \emph{a priori} become bigger. Our main point is that if we shrink $S$ is an appropriate way, then the Tate-Shafarevich group will stay small. %small. 
This is based on the properties of the so-called \emph{stable} and \emph{persistent} sets of primes studied in \cite{Ivanov_stable_sets}. In \S\ref{sec:stable_set} we also give new examples of such sets with particularly nice properties.

% Here we give an isomorphism of this group with a modified Tate-Shafarevich group defined by a $p$-stable set. 
At the end of the paper we state a generalization of a classical question posed by Cassels about the $p$-divisibility of 
elements of the Tate-Shafarevich group and we discuss the implications of our results for such a question.

\subsection*{Notation}
We denote by $p$ a rational prime. We fix once and for all an algebraic closure $\overline \bQ$ of $\bQ$, and consider all algebraic extensions of $\bQ$ as subfields of $\overline \bQ$. Unless stated otherwise, $k$ always denote a number field of degree $d = [K:\bQ]$. If $\ell/k$ is a Galois extension, we denote by $\Gal_{\ell/k}$ its Galois group.
% Moreover, we fix an algebraic closure $\overline k$ of $k$ and consider all algebraic extensions of $k$ to lie in $\overline k$. 
We denote by $\Sigma_k$ the set of primes of $k$. If $S \subseteq \Sigma_k$ and $\ell/k$ is a finite extension, we write $S_\ell$ for the preimage of $S$ under the natural map $\Sigma_\ell \rar \Sigma_k$. We denote by $\delta_k(S) \in [0,1]$ the Dirichlet density of a set $S \subseteq \Sigma_k$, whenever it exists. Whenever we write ``density'' below, we mean ``Dirichlet density''. 

By $A$ we denote a commutative algebraic group defined over $k$ and by $K_n$ the extension of $k$ trivializing $A[p^n]$, i.e.\ the $p^n$-division field of $A$ over $k$, where $p$ is a fixed prime and $n$ is a positive integer.

% \par\color{blue} I like the Notation here, however a referee once told me that in general the notation should not be
% given in the introduction. I want to signal this, even if I agree to leave the Notation here\normalcolor

\subsection*{Acknowledgements} The first author was supported by a Heisenberg
grant of the DFG (grant nr. IV 177/3-1), based at the Universities of Bonn
and Bochum. Also he was supported by the Leibniz Universit\"at Hannover. The second author is a member of INdAM-GNSAGA.

\section{Testing local-global divisibility at a stable set}

\subsection{Review of stable sets}
We recall the following definition.

\begin{Def}[\cite{Ivanov_stable_sets}, \S2]\label{def:recall_stable_sets}
Let $\cL/k$ be an algebraic extension, $S \subseteq \Sigma_k$ and $\lambda > 1$.
\begin{itemize}
\item[(1)] A finite subextension $\cL/K/k$ is called \emph{$\lambda$-stabilizing} (resp. \emph{persisting}) for $S$, if there exists a subset $T \subseteq S$ and some $a \in (0,1]$ such that for all finite subextensions $\cL/L/K$ one has $\lambda a > \delta_L(T_L) \geq a$ (resp. $\delta_L(T_L) = \delta_K(T_K) > 0$).
\item[(2)] $S$ is \emph{$\lambda$-stable} (resp. \emph{persisting}) for $\cL/k$, if it has a $\lambda$-stabilizing (resp. persisting) subextension of $\cL/k$.
\end{itemize}
\end{Def}

There are many natural examples of stable and persistent sets, cf. \cite[\S3]{Ivanov_stable_sets}. For example, if $\ell/k$ is a finite Galois extension, then for $\sigma \in \Gal_{\ell/k}$ the set 
\begin{equation}\label{eq:chebotarev_set}
P_{\ell/k}(\sigma) = \{ \fp \in \Sigma_k \colon \fp \text{ is unramified in $\ell/k$ and } \Frob_{\fp} = C(\sigma,\Gal_{l/k}) \},
\end{equation}
where $C(\sigma,G_{\ell/k})$ denotes the conjugacy class of $\sigma$, is persistent for any algebraic extension $\cL/k$ satisfying $C(\sigma,G_{\ell/k}) \cap \Gal_{\ell/ \cL\cap \ell} \neq \varnothing$, with persisting field $\cL\cap \ell$. In \S\ref{sec:stable_set} we give new interesting examples of persistent sets.

\smallskip

Stable sets generalize sets of density one in the following sense. Let $\ell/k$ be a finite extension. If $S$ is a set of primes of $k$ with density one, then any element of $\Gal_{\ell/k}$ is a Frobenius at $S$, and consequently any cyclic subgroup is a decomposition subgroup of a prime in $S$. Weakening the assumption on $S$ to be $p$-stable for $\ell/k$ with stabilizing field $k$, destroys the claim about \emph{elements}, but the claim about \emph{cyclic $p$-subgroups} (that is, of $p$-power order and not just of order $p$) remains true:

% The claim about elements becomes false, if $S$ is only assumed to be $p$-stable for $\ell/k$ with stabilizing field $k$. However, the claim about cyclic $p$-subgroups continues to hold, as the following result says.

\begin{lm}[\cite{Ivanov_stable_sets}, Lemma 4.4]\label{lm:main_lemma_about_stable_sets}
Let $\ell/k$ be a finite Galois extension, $S$ a set of primes of $k$ and $p$ a rational prime such that $S$ is $p$-stable for $\ell/k$ with $p$-stabilizing field $k$. Then any cyclic $p$-subgroup of $\Gal_{\ell/k}$ is the decomposition subgroup of a prime in $S$. 
\end{lm}

\subsection{$p$-stable sets detect local-global divisibility by $p^n$} { } \hspace{0.1cm}

\par\smallskip Recall the definition of the cohomology group satisfying the \emph{local conditions} \cite{DvornicichZ_01}. Let $\Gamma$ be a finite group
and $M$ a discrete $\Gamma$-module. Then

\begin{equation}\label{eq:cohomology_with_local_conditions}
H^1_{\rm loc}(\Gamma,M) := \ker\left( H^1(\Gamma,M) \rar \prod_{C\subseteq \Gamma} H^1(C,M) \right),
\end{equation}
where the product is taken over all cyclic subgroups $C \subseteq \Gamma$, and the map is the product of restriction maps. 
% Such a group is called first cohomology group.

\begin{lm}\label{lm:local_conditions_only_psubgroups}
Let $\Gamma,M$ be as above, and suppose that $M$ is $p$-primary for a rational prime $p$. Then $H^1_{\rm loc}(\Gamma,M) = \ker\left( H^1(\Gamma,M) \rar \prod_{C\subseteq \Gamma} H^1(C,M) \right)$, where the product is taken over all cyclic $p$-subgroups $C \subseteq \Gamma$.
% \eqref{eq:cohomology_with_local_conditions} remains to hold if one takes the product only over all all cyclic $p$-subgroups $C$.
\end{lm}
\begin{proof}
This follows from the fact that for any finite group $H$  and any $p$-primary module $M$, $H^1(H,M) \rar H^1(H_p,M)$ is injective, where $H_p$ is a $p$-Sylow subgroup of $H$ (see \cite[(1.6.10)]{NSW}). 
\end{proof}

% Note that if $M$ is $p$-primary, then one may equivalently take the product over all cyclic $p$-subgroups.

Let $K_n^{\rm ab}(p)/k$ be the maximal abelian pro-$p$ extension of $K_n$. The following generalization of \cite[Prop.~2.1]{DvornicichZ_01} shows that it suffices to test local-global divisibility by $p^n$ at a $p$-stable set of primes:  

\begin{prop}\label{prop:p_stable_sets_test_p_divisibility}
Let $p$ be a rational prime and $n \geq 1$ an integer. Let $A/k$ be a commutative algebraic group.  Let $S$ be a set of primes of $k$, which is $p$-stable for  $K_n^{\rm ab}(p)/k$ with $p$-stabilizing field $k$. Assume that $H^1_{\rm loc}(\Gal_{K_n/k}, A[p^n]) = 0$. Then the following holds. Let $P \in A(k)$, such that for all $\fp \in S$, there is some $Q_\fp \in A(k_\fp)$ with $P_{\fp} = p^n Q_{\fp}$. Then there is some $Q \in A(k)$ with $P = p^n Q$.

% Let $p$ be a rational prime. Let $S$ be a set of primes of a number field $k$ which is $p$-stable for $k(p)/k$, with $p$-stabilizing field $k$. Let $A$ be a commutative algebraic group over $k$. Let $n \geq 1$ and let $K_n/k$ be the finite extension defined by trivializing $A[p^n]$. Assume that $H^1_{\rm loc}(K/k, A[p^n]) = 0$. Then the following holds: let $P \in A(k)$, such that for all $\fp \in S$, there is some $Q_\fp \in A(k_\fp)$ with $P_{\fp} = p^n Q_{\fp}$. Then there is some $Q \in A(k)$ with $P = p^n Q$.
\end{prop}

\begin{proof}
Let $D \in A(\overline \bQ)$ be a point with $p^n D = P$ and let $\ell=k(D)$ be the corresponding extension of $k$. Put $F = K_n\cdot \ell$. Then $F/k$ is Galois and $\ell/k$ is cyclic of $p$-power degree, so in particular $F \subseteq K_n(p)^{\rm ab}$. One can define a $1$-cocycle $c \colon \Gal_{F/k} \rar A[p^n]$, by $c(\sigma):=\sigma(D)-D$, for all $\sigma\in \Gal_{F/k}$. 
Its image $[c] \in H^1(\Gal_{F/k}, A[p^n])$ is zero if and only if $P = p^n D'$ for some $D' \in A(k)$, see \cite[p.~320]{DvornicichZ_01}. Moreover, as by assumpion $P$ is locally $p^n$-divisible at any $\fp \in S$, the same argument with cocycles
show that the restriction of $[c]$ to $H^1(C,A[p^n])$ is zero, where $C \subseteq \Gal_{F/k}$ is the decomposition subgroup of any prime in $S$. Now, as $F \subseteq K_n(p)^{\ab}$, by Lemma \ref{lm:main_lemma_about_stable_sets} the set of decomposition subgroups in $\Gal_{F/k}$ at primes in $S$ contains the set of all cyclic $p$-subgroups of $\Gal_{F/k}$, and so by Lemma \ref{lm:local_conditions_only_psubgroups} we deduce $[c] \in H^1_{\rm loc}(\Gal_{F/k},A[p^n])$. To finish the proof of Proposition \ref{prop:p_stable_sets_test_p_divisibility} it thus remains to show that the restriction via $\Gal_{F/k} \tar \Gal_{K_n/k}$ induces an isomorphism $H^1_{\rm loc}(\Gal_{K_n/k},A[p^n]) \stackrel{\sim}{\rar} H^1_{\rm loc}(\Gal_{F/k},A[p^n])$. But this is done in the proof of \cite[Prop. 2.1]{DvornicichZ_01}.
\end{proof}

\begin{cor}\label{cor:K_n-rational_divisor}
Let $p$ be a rational prime and $n \geq 1$ an integer. Let $A/k$ be a commutative algebraic group. Let $S$ be a set of primes of $k$, which is $p$-stable for  $K_n^{\rm ab}(p)/k$ with $p$-stabilizing field $k$. Assume that $P \in A(k)$, such that for all $\fp \in S$, there is some $Q_\fp \in A(k_\fp)$ with $P_{\fp} = p^n Q_{\fp}$. Then 
$D\in A(K_n)$, for all $D$ such that $P=p^nD$.
\end{cor}

\begin{proof}
One can apply the same argument used in the proof of \cite[Corollary 2.3]{DvornicichZ_01}, 
by substituting  the set of density one with the set $S$.
% \color{blue}  Otherwise we can state:\normalcolor One can adapt the proof of
% \cite[Corollary2.3]{DvornicichZ_01}, but we state the argument here to mantain the
% paper more self-contained. We denote by $K_{n,\fq}$ the completion of $K_n$ at a prime
% $\fq$. Observe that in particular $P\in A(K_n)$ and, by assumptions, for every prime $\fq$ of $K_n$
% there exists $D_\fq\in A(K_{n,\fq})$ such that $P=p^nD_\fq$. Since $H^1_{\loc}(\Gal_{K_n/K_n},A[p^n])=0$,
% because of $\Gal_{K_n/K_n}$ being trivial, by Proposition \ref{prop:p_stable_sets_test_p_divisibility} 
% there exists $D\in A(K_n)$ such that $P=p^nD$. Let
% $\tilde{D}\in A(\bar{k})$ such that $P=p^n\tilde{D}$. Every two $p^n$-divisors of $P$ differ by
% a $p^n$-torsion point, then $D-\tilde{D} \in A(K_n)$, implying $\tilde{D}\in A(K_n)$.
\end{proof}

We denote by $G_k$ the absolute Galois group $\Gal_{\bar{\bQ}/k}$ and by $G_{k_{\mfp}}$
the absolute Galois group $\Gal_{\bar{k}_{\mfp}/k_{\mfp}}$, where $\bar{k}_{\mfp}$ is an
algebraic closure of $k_\mfp$. 
Let $S$ be a subset of primes of $k$ which is $p$-stable at $K_n^{\rm ab}(p)/k$ with $p$-stabilizing field $k$. We define a modified Tate-Shafarevich
group related to $S$ and we will prove that it is isomorphic to $H^1_{\loc}(\Gal_{K_n/k},A[p^n])$.

\begin{Def} Let $A$ be a comutative algebraic group defined over a number field $k$ and let $S$ be a set of primes of $k$, unramified in $K_n$, which is $p$-stable at $K_n^{\rm ab}(p)/k$ with $p$-stabilizing field $k$. We denote by $\Sha_{S}(k,A[p^n])$ the subgroup of $H^1(G_k,A[p^n])$ formed by the
classes of the cocycles vanishing in $k_{\mfp}$, for all $\mfp\in S$, i.e.\

\begin{equation}\label{Sha}
\Sha_{S}(k,A[p^n]):=\bigcap_{\mfp\in S} \ker ( H^1(G_k,A[p^n])\xrightarrow{\makebox[1cm]{{\small $res_\mfp$}}} H^1(G_{k_{\mfp}},A[p^n])) \end{equation}
\end{Def}

Notice that in Equation \eqref{Sha} by replacing $S$ with the set of primes of $k$ one gets the definition of the classical
Tate-Shafarevich group $\Sha(k,A[p^n])$.
We are going to prove that $H^1_{\rm loc}(G, A[p^n])\simeq \Sha_{S}(k, A[p^n])$.

\begin{prop}\label{thm:p_stable_sets_test_p_divisibility}
Let $p$ be a rational prime and $n \geq 1$ an integer. Let $A/k$ be a commutative algebraic group. Let $S$ be a set of primes of $k$, unramified in $K_n$, which is $p$-stable for  $K_n^{\rm ab}(p)/k$ with $p$-stabilizing field $k$. Then $H^1_{\rm loc}(G, A[p^n]) \simeq \Sha_{S}(k,A[p^n])$.
\end{prop}

\begin{proof}
Let $S_{K_n}$ denote the set of primes $w_{\mfp}$ of $K_n$ extending the primes $\mfp$ in $S$ and by $K_{n,w_{\mfp}}$ we denote
the completion of $K_n$ at the place $w_{\mfp}$. Let $G_{\mfp}:= \Gal_{K_{n,w_{\mfp}}/k_{\mfp}}$ and consider the following diagram given by the inflation restrictions exact sequence

\[
\begin{array}{ccccccc}
0 & \xrightarrow{\hspace{0.4cm}} & H^1(G,A[p^n])  & \xrightarrow{\hspace{0.1cm}inf\hspace{0.1cm}} &  H^1(G_k,A[p^n]) &   \xrightarrow{\hspace{0.1cm}res\hspace{0.1cm}} & H^1(G_{K_n},A[p^n]) \\
& &\hspace{0.8cm} \Big\downarrow {\prod res{_\mfp}}  & &\hspace{0.8cm}\Big\downarrow {\prod res{_\mfp}} & &\hspace{0.8cm}\Big\downarrow {\prod res_{w_\mfp}} \\
0 & \xrightarrow{\hspace{0.4cm}} &\prod_{\mfp\in S} H^1(G{_\mfp},A[p^n])  & \xrightarrow{\hspace{0.1cm}inf\hspace{0.1cm}} &  \prod_{\mfp\in S}H^1(G_{k{_\mfp}},A[p^n]) &  \xrightarrow{\hspace{0.1cm}res\hspace{0.1cm}} &\prod_{{w_\mfp}\in S_{K_n}} H^1(G_{K_{n,w_\mfp}},A[p^n]) \\
\end{array} 
\]

\bigskip\noindent  The kernel of the vertical map on the left is  $H^1_{\loc}(G,A[p^n])$ by assumption on $S$ and by Lemma \ref{lm:main_lemma_about_stable_sets} and the kernel of the
central vertical map is $\Sha_{S}(k,A[p^n])$. The
vertical map on the right is injective because of $G_{K_n}$ acting trivially on $A[p^n]$ and because of $G_{K_{n,w}}$
varying over all cyclic subgroups of $G_{K_n}$ as $w$ varies in $S_{K_n}$ by Lemma \ref{lm:main_lemma_about_stable_sets}. 
%We have such a diagram.
%
%
%\[
%\begin{array}{ccccccc}
%0 &\xrightarrow{\hspace{0.4cm}} & H^1_{\loc}(G,A[p^n]) &  \xrightarrow{\hspace{0.1cm}inf\hspace{0.1cm}} & \Sha_{S}(k,A[p^n])   & \xrightarrow{\hspace{0.1cm}res\hspace{0.1cm}} & 0 \\
%& & \Big\downarrow  & &\Big\downarrow  & &\Big\downarrow \\
%0 & \xrightarrow{\hspace{0.4cm}} & H^1(G,A[p^n])  & \xrightarrow{\hspace{0.1cm}inf\hspace{0.1cm}} &  H^1(G_k,A[p^n]) &   \xrightarrow{\hspace{0.1cm}res\hspace{0.1cm}} & H^1(G_{K_n},A[p^n]) \\
%& &\hspace{0.8cm} \Big\downarrow {\prod res{_\mfp}}  & &\hspace{0.8cm}\Big\downarrow {\prod res{_\mfp}} & &\hspace{0.8cm}\Big\downarrow {\prod res_{w_\mfp}} \\
%0 & \xrightarrow{\hspace{0.4cm}} &\prod_{\mfp\in S} H^1(G{_\mfp},A[p^n])  & \xrightarrow{\hspace{0.1cm}inf\hspace{0.1cm}} &  \prod_{\mfp\in S}H^1(G_{k{_\mfp}},A[p^n]) &  \xrightarrow{\hspace{0.1cm}res\hspace{0.1cm}} &\prod_{{w_\mfp}\in S_{K_n}} H^1(G_{K_{n,w_\mfp}},A[p^n]) \\
%\end{array}
%\]
%
Therefore $H^1_{\loc}(G,A[p^n])\simeq \Sha_{S}(k,A[p^n])$. \qedhere
\end{proof}

\medskip

At cost of slightly strengthening the assumption on $S$, we can eliminate the dependence on the particular algebraic group $A$ in Proposition \ref{prop:p_stable_sets_test_p_divisibility}. For an integer $M > 0$, consider the class $\fC_{\leq M}(k)$ of all commutative algebraic groups over $k$, such that $\dim_{\bF_p}\#A[p](\overline \bQ) \leq M$. Moreover, for a set $\bP \subseteq \Sigma_\bQ$ of rational primes, let $k(\bP)$ be the compositum of all finite extensions of $k$ whose orders are products of elements of $\bP$.

\begin{cor}\label{cor:p_stable_sets_test_p_divisibility}
Let $p$ be a rational prime and let $M>0$. Let $S$ be a set of primes of $k$. Assume that $S$ is $p$-stable with $p$-stabilizing field $k$ for the extension $k(\bP(M))/k$, where $\bP(M)$ is the set of all prime divisors of $\#\GL_{M'}(\bF_p)$ for all $M' \leq M$. Then for any $A \in \fC_{\leq M}(k)$, such that $H^1_{\rm loc}(\Gal_{K_n/k}, A[p^n]) = 0$, the following holds.

Let $P \in A(k)$, such that for all $\fp \in S$, there is some $Q_\fp \in A(k_\fp)$ with $P_{\fp} = p^n Q_{\fp}$. Then there is some $Q \in A(k)$ with $P = p^n Q$.
\end{cor}

For example, if one restricts further to the class of elliptic curves one can take $M = 2$ and $\bP$ the set of prime divisors of $p,p\pm 1$. For
abelian varieties of dimension $\leq g$, one can take $M=2g$.

\begin{proof}[Proof of Corollary \ref{cor:p_stable_sets_test_p_divisibility}]
We have to show that for a particular $A \in \fC_M(k)$, the assumptions of Proposition \ref{prop:p_stable_sets_test_p_divisibility} hold, i.e., that $K_n(p)^{\ab} \subseteq k(\bP(M))$. As $p \in \bP(M)$, it suffices to show that $K_n \subseteq k(\bP(M))$. Also as $K_n/K_1$ is a $p$-extension, it suffices to show that $K_1/k \subseteq k(\bP(M))$. But $\Gal_{K_1/k} \subseteq  \GL(A[p])$, which is a subgroup of $\GL_{M'}(\bF_p)$ for some $M'\leq M$.
\end{proof}

Strengthening assumptions on $S$ even further, Corollary \ref{cor:p_stable_sets_test_p_divisibility} immediately gives the following:

\begin{cor}\label{cor:for_absolutely_stable_S} 
Fix a prime $p$. Let $S$ be a set of primes of $k$, which is persistent for $\overline \bQ/k$ with persisting field $k$. For any commutative algebraic group $A/k$ and any $n > 0$, if $H^1_{\rm loc}(\Gal_{K_n/k},A[p^n]) = 0$, the following holds.

Let $P \in A(k)$, such that for all $\fp \in S$, there is some $Q_\fp \in A(k_\fp)$ with $P_{\fp} = p^n Q_{\fp}$. Then there is some $Q \in A(k)$ with $P = p^n Q$.
\end{cor}

We show in Proposition \ref{prop:absolutely_persistent_set} below that sets $S$ satisfying the requirements of the corollary exist in abundance.
Using this along with existing results on the original form of Problem \ref{problem:DZ}, Corollary \ref{cor:for_absolutely_stable_S} specializes to the case of elliptic curves:

\begin{proof}[Proof of Theorem \ref{thm:small_test_set}]
Pick a set $S \subseteq \Sigma_\bQ$ as constructed in Proposition \ref{prop:absolutely_persistent_set}. Let $E/k$ be an elliptic curve. By Corollary \ref{cor:for_absolutely_stable_S} it suffices to show that  $H^1_{\rm loc}(\Gal_{K_p/k},E[p]) = 0$ for all $p$, resp. that $H^1_{\rm loc}(\Gal_{K{p^n}/k},E[p^n]) = 0$ for all $p>B(d)$ and all $n > 1$. The first case is easy, as was observed in \cite[beginning of \S 3]{DvornicichZ_01}: $\Gal_{K_p/k}$ is then a subgroup of $\GL_2(\bF_p)$, which implies that the $p$-Sylow subgroup of $\Gal_{K_p/k}$ is cyclic and so $H^1_{\rm loc}(\Gal_{K_p/k},E[p]) = 0$. In the second case we may apply \cite[Theorem 1' (on p. 8)]{PaladinoRV_12} (see also Corollary 2 of \emph{loc.\ cit.}) and \cite{PaladinoRV_14}. 
\end{proof}

\section{New examples of stable sets} \label{sec:stable_set}

It is easy to give examples of stable sets $S$ with arbitrary small density in the whole tower $\cL/k$, when $\cL$ is some reasonably small subextension of $\overline\bQ/k$. However, those examples will often not be stable for other towers $\cL'/k$. Consider, for example, $S = P_{\ell/k}(\sigma)$ as in \eqref{eq:chebotarev_set}. If $\sigma = 1$, then $S$ will be stable --even persistent-- for any extension $\cL/k$, but if $\cL \supseteq \ell$, the persisting field is $\ell$ and $\delta_\ell(S_\ell) = 1$, that is $S_\ell$ eventually becomes ``big''. On the other hand, if $\sigma \neq 1$, then $\delta_{\ell'}(S_{\ell'}) = 0$ for any finite $\ell'/\ell$, and hence $S$ is not stable for $\cL/k$ whenever $\cL \supseteq \ell$. With this in mind, we now produce now many examples of sets persistent for $\overline\bQ/k$ with persisting field $k$ and arbitrary small positive density.

\begin{prop}\label{prop:absolutely_persistent_set}
For any number field $k$ and any $\varepsilon > 0$, there exists a set $S$ of primes of $k$ satisfying 
\[ 
0 < \delta_k(S) = \delta_\ell(S_\ell) < \varepsilon
\] 
for all finite extensions $\ell/k$. In particular, $S$ is persistent for $\overline\bQ/k$ with persisting field $k$.
\end{prop}
\begin{proof}
Let $p$ be a prime such that $\frac{1}{p-1} < \varepsilon$. Let $k_\infty/k$ be a $\bZ_p$-extension (e.g.\ the cyclotomic one) with Galois group identified with $\bZ_p$. For $n\geq 1$, choose $a_n \in \bF_p^\times$ and consider the set
\[
A = \bigcup_{n\geq 1} \left(a_n p^{n-1} + p^n \bZ_p\right) \subseteq \bZ_p.
\]
We put 
\[S = P_{k_\infty/k}(A),\] 
the set of all primes unramified in $k_\infty/k$, whose Frobenius lies in $A$. Clearly, $A$ is open in $\bZ_p$ and one computes $\overline A = A \cup \{0\}$. Equip $\bZ_p$ with the invariant Haar measure $\mu$ normalized such that $\mu(\bZ_p) = 1$. Then the boundary $\overline A \sm A^\circ = \{0\}$ of $A$ has measure $0$, and the infinite Chebotarev theorem \cite[I.2.2 Corollary 2b)]{Serre_abelian} (which we may apply as $k_\infty/k$ is ramified at most in the finitely many primes above $p$ and $\infty$) then shows that $S$ has a density and that it is equal to
\[
\delta_k(S) = \mu(A) = \sum_{n\geq 1}^\infty p^{-n} = \frac{1}{p-1}.
\]

\smallskip

Now, let $\ell/k$ be a finite extension. Then there is some $m \geq 0$ such that $\ell \cap k_\infty = k_{m} := (k_\infty)^{p^{m}\bZ_p}$. Let $\ell_\infty = k_\infty . \ell$. Via $\Gal_{\ell_\infty/\ell} \cong \Gal_{k_\infty/k_{m}}$, we may identify $\Gal_{\ell_\infty/\ell}$ with $p^{m}\bZ_p \subseteq \bZ_p$. Let ${\rm Spl}_{\ell/k}$ denotes the set of primes of $\ell$, which are split (and unramified) over $k$. Then 
\begin{equation}\label{eq:formula_for_pullback_of_primes}
S_\ell \cap {\rm Spl}_{\ell/k} = P_{\ell_\infty/\ell}(A \cap p^{m}\bZ_p),
\end{equation}
where we ignore the finitely many primes of $\ell$ which ramify in $\ell_\infty/k$.
% (indeed, as $S$ is an (infinite) union of Chebotarev sets, it suffices to check \eqref{eq:formula_for_pullback_of_primes} for $k_\infty, \ell_\infty$ replaced by the finite subextensions $k_n,\ell_n$, where it is easy).
Now, $\Sigma_\ell \sm {\rm Spl}_{\ell/k}$ consists of primes of $\ell$, which are not split over $\bQ$, so it has density $0$. In particular, $\delta_\ell(S_\ell)$ exists if and only if $\delta_\ell(S_\ell \cap {\rm Spl}_{\ell/k})$ exists, in which case both agree. On the other hand, the argument using infinite Chebotarev applied above to compute $\delta_k(S)$ applies also to $P_{\ell_\infty/\ell}(A \cap p^{m}\bZ_p)$, giving $\delta_\ell(P_{\ell_\infty/\ell}(A \cap p^{m}\bZ_p)) = \frac{1}{p-1}$. Combining the two computations, we get $\delta_\ell(S_\ell) = \delta_\ell(P_{\ell_\infty/\ell}(A \cap p^{m}\bZ_p)) = \frac{1}{p-1}$, finishing the proof. \qedhere

% Let $p$ be a prime such that $\frac{1}{p-1} < \varepsilon$. Let $M_\infty/k$ be a $\bZ_p$-extension (e.g. the cyclotomic one). For $n \geq 0$, write $M_n/k$ for the unique subextension of degree $[M_n : k] = p^n$. For any $n\geq 1$, let $1 \neq \sigma_n \in \Gal_{M_n/M_{n-1}) \subseteq \Gal_{M_n/k)$. Put 
% \[
% S = \bigcup_{n\geq 1} P_{M_n/k}(\sigma_n)
% \]
% It is clear that $\delta_k(P_{M_n}/k(\sigma_n)) = p^{-n}$ for all $n \geq 1$. So, by the $\sigma$-additivity of the density, $\delta_k(S) = \sum_{n \geq 1} p^{-n} = \frac{1}{p-1}$. Now let $\ell/k$ be any finite extension. There is some $n_0 \geq 0$ with $\ell \cap M_{\infty} = M_{n_0}$. By \cite[Prop. 3.3]{Ivanov_stable_sets} we have 
% \[
% \delta_\ell(P_{M_n/k}(\sigma_n)_\ell) = \frac{\#\left( \{\sigma_n \} \cap \Gal_{M_n/M_n \cap M_{n_0})\right)}{\Gal_{M_n/M_n \cap M_{n_0})}.
% \]
% If $n\leq n_0$, $\Gal_{M_n/M_n \cap M_{n_0}) = 1$ and as $\sigma_n \neq 1$, we get $\delta_\ell(P_{M_n/k}(\sigma_n)_\ell) = 0$. If $n > n_0$, $\Gal_{M_n/M_n \cap M_{n_0}) = \Gal_{M_n/M_{n_0})$ and the formula shows $\delta_{\ell}(P_{M_n/k}(\sigma_n)_\ell) = p^{-(n-n_0)}$. This implies
% \[
% \delta_\ell(S_\ell) = \sum_{n \geq 1} \delta_\ell(P_{M_n/k}(\sigma_n)_\ell) = \sum_{n > n_0} p^{n-n_0} = \frac{1}{p-1} = \delta_k(S),
% \]
% finishing the proof.
\end{proof}

\begin{rem}
One has to be careful in the above proof, as the Dirichlet density does not satisfy $\sigma$-additivity: suppose that $T_n \subseteq \Sigma_k$ ($n \geq 1$) is a collection of mutually disjoint subsets, such that $\delta_k(T_n)$ exists. Let $T = \bigcup_n T_n$. Then it might happen that $\delta_k(T)$ does not exist, and even if it exists, it might happen that $\delta_k(T) \neq \sum_{n\geq 1} \delta_k(T_n)$ (it is enough to consider singletons $T_n = \{\fp_n\}$ for any $n$). However, by the argument in the proof of Proposition \ref{prop:absolutely_persistent_set}, the density of the set $S$, which is in fact a disjoint union of infinitely many Chebotarev sets, exists and is equal to the sum of densities of these Chebotarev sets.
% The set $S$ constructed in the proof of Proposition \ref{prop:absolutely_persistent_set} can also be written as an infinite disjoint union of Chebotarev sets, $S = \bigcup_{n\geq 1}P_{k_n/k}(\bar a_n)$, where $k_n/k$ is the subextension of $k_\infty/k$ of degree $p^n$ and $\bar a_n \in \bZ/p^n\bZ$ is the image of $a_n$. The Dirichlet density obviously does not satisfy 
\end{rem}

% \AI{One could try to prove a more general result, --from which Proposition \ref{prop:absolutely_persistent_set} follows as a corollary-- like: ``if $G$ is a profinite group and $A \subseteq G$ an open subset, then $\mu(\delta A)= 0$'' ($\delta A = \overline A \sm A^\circ$ is the boundary). This would imply something like: ``for any collection of Chebotarev sets $P_{\ell_i/k}(\sigma_i)$, such  that $\prod_i \ell_i/k$ is unramified outside a finite set of primes, the union of all $P_{\ell_i/k}(\sigma_i)$ has a density, which is equal to the sum of densities of the single setS. However, for the Lebeque measure on $\bR^n$, it seems that there are a lot of examples of open sets with boundary of positive measure. May be ``$A$ open'' can be replaced by sth stronger?}

\section{On a question posed by Cassels}

In addition Problem \ref{problem:DZ} is strongly related to the following question stated by Cassels in 1962 that remained open for 50 years (see \cite{DP_22} for further details).

\begin{cass} 
Let $k$ be a number field and $E$ be an elliptic curve defined over $k$. Are the
elements of $\Sha(k,E)$ infinitely divisible by a prime $p$ when considered as elements of the Weil-Ch\^{a}telet group
$H^1(G_k,A)$ of all classes of principal homogeneous spaces for $E$ defined over $k$?
\end{cass}

An affirmative answer for $p>B(d)$
is implied by \cite[Theorem 3]{Cre2} and the results in \cite{PaladinoRV_12,PaladinoRV_14}. 
Since 1972 this question was considered in abelian varieties of every dimension by various authors \cite{Bas, Cre2, CS1}.
In particular, Creutz showed sufficient and necessary conditions to get an affirmative answer and
 the existence of counterexamples for every $p$ in infinitely many abelian varieties \cite{Cre2}. \cite{Cre}. Moreover,
\c{C}iperiani and Stix also showed sufficient conditions to get an affirmative answer \cite{CS1}.
\par In the spirit of this article, we can pose the following more general question. 

\begin{problem}\label{prob2}
Let $k$ be a number field and $A$ an abelian variety defined over $k$. Let $S$ be an infinite set of places
of $k$ and let
\[ 
\Sha_{S}(k,A):=\bigcap_{\mfp\in S} \ker ( H^1(G_k,A[p^n])\xrightarrow{\makebox[1cm]{{\small $res_\mfp$}}} H^1(G_{k_{\mfp}},A[p^n])).
\]
Are the elements of $\Sha_{S}(k,A)$ infinitely divisible by a prime $p$ when considered as elements of the Weil-Ch\^{a}telet group
$H^1(G_k,A)$?
\end{problem}

As a consequence of Proposition \ref{thm:p_stable_sets_test_p_divisibility}, in the case of elliptic curves
curves Problem \ref{prob2} has an affirmative answer for all $p>B(d)$, for every sets $S$ which is
$p$-stable for  $\cup_{n\in \bN} K_n^{ab}(p)/k$ with $p$-stabilizing field $k$. 

\begin{cor} \label{cor_Cass}
Let $k$ be a number field and $E$ an elliptic curve over $k$. Let $S$ be a set of places of $k$ which
is $p$-stable for $\cup_{n\in \bN} K_n^{ab}(p)/k$ with stabilizing field $k$. Then the elements of $\Sha_{S}(k,E)$ are infinitely divisible by every 
$p>B(d)$ when considered as elements of $H^1(G_k,E)$.
\end{cor}

\begin{proof}
By the results in \cite{PaladinoRV_12} and \cite{PaladinoRV_14}, we have $H^1_{\loc}(G,E[p^n])=0$,
for all $p>B(d)$ and all $n\geq 1$.
By Proposition \ref{thm:p_stable_sets_test_p_divisibility}, we have $\Sha_{S}(k,E[p^n])=0$, 
for all $p>B(d)$ and all $n\geq 1$. The conclusion is then implied by \cite[Theorem 3]{Cre2}.
\end{proof}

\noindent Observe that by Proposition \ref{prop:absolutely_persistent_set} there are many sets $S$ of primes satisfying the
assumtpions of Corollary \ref{cor_Cass}.

%\begin{cor}
%Let $k$ be a number field and $A$ an elliptic curve over $k$. Consider the extension
%$\cup_{n\in \bN} K_n^{ab}(p)$  and assume that $S$ is a set of places of $k$ which
%is persistent for the subextension $K_n/k$, for some $n$. \color{red} Is this correct?
%Are there improvement here?.\normalcolor Then the elements of $\Sha_{S}(k,A)$ are infinitely divisible by every 
%$p>B(d)$ in $H^1(G_k,A)$.
%\end{cor}

% \printbibheading
% \printbibliography

\bibliography{bib_ADLV}{}

\providecommand{\bysame}{\leavevmode\hbox to3em{\hrulefill}\thinspace}
\providecommand{\MR}{\relax\ifhmode\unskip\space\fi MR }
\providecommand{\MRhref}[2]{%
  \href{http://www.ams.org/mathscinet-getitem?mr=#1}{#2}
}
\providecommand{\href}[2]{#2}
\begin{thebibliography}{NSW13}

\bibitem[Ba{\v s}72]{Bas}
M.~I. Ba{\v s}makov, \emph{Cohomology of {A}belian varieties over a number
  field}, Uspehi Mat. Nauk \textbf{27} (1972), no.~6(168), 25--66. \MR{0399110}

\bibitem[Cre12]{Creutz_12}
Brendan Creutz, \emph{A {G}runwald-{W}ang type theorem for abelian varieties},
  Acta Arith. \textbf{154} (2012), no.~4, 353--370. \MR{2949874}

\bibitem[Cre13]{Cre2}
\bysame, \emph{Locally trivial torsors that are not {W}eil-{C}h\^{a}telet
  divisible}, Bull. Lond. Math. Soc. \textbf{45} (2013), no.~5, 935--942.
  \MR{3104985}

\bibitem[Cre16]{Cre}
\bysame, \emph{On the local-global principle for divisibility in the cohomology
  of elliptic curves}, Math. Res. Lett. \textbf{23} (2016), no.~2, 377--387.
  \MR{3512890}

\bibitem[{\c C}S15]{CS1}
Mirela {\c C}iperiani and Jakob Stix, \emph{Weil-{C}h\^{a}telet divisible
  elements in {T}ate-{S}hafarevich groups {II}: {O}n a question of {C}assels},
  J. Reine Angew. Math. \textbf{700} (2015), 175--207. \MR{3318515}

\bibitem[DP22a]{DP_22}
Roberto Dvornicich and Laura Paladino, \emph{Local-global questions for
  divisibility in commutative algebraic groups}, Eur. J. Math. \textbf{8}
  (2022), no.~suppl. 2, S599--S628. \MR{4507256}

\bibitem[DP22b]{DvornicichP_21}
\bysame, \emph{On the division fields of an elliptic curve and an effective
  bound to the hypotheses of the local-global divisibility}, Int. J. Number
  Theory \textbf{18} (2022), no.~7, 1567--1590. \MR{4439046}

\bibitem[DZ01]{DvornicichZ_01}
R.~Dvornicich and U.~Zannier, \emph{Local-global divisibility of rational
  points in some commutative algebraic groups}, Bull. Soc. Math. France
  \textbf{129} (2001), no.~3, 317--338.

\bibitem[DZ07]{DvornicichZ_07}
Roberto Dvornicich and Umberto Zannier, \emph{On a local-global principle for
  the divisibility of a rational point by a positive integer}, Bull. Lond.
  Math. Soc. \textbf{39} (2007), no.~1, 27--34. \MR{2303515}

\bibitem[Ill08]{Illengo_08}
Marco Illengo, \emph{Cohomology of integer matrices and local-global
  divisibility on the torus}, J. Th\'{e}or. Nombres Bordeaux \textbf{20}
  (2008), no.~2, 327--334. \MR{2477506}

\bibitem[Iva16]{Ivanov_stable_sets}
A.~Ivanov, \emph{Stable sets of primes in number fields}, Algebra \& Number
  Theory \textbf{10} (2016), no.~1, 1--36.

\bibitem[Mer96]{Merel_96}
L.~Merel, \emph{Bornes pour la torsion des courbes elliptiques sur les corps de
  nombres}, Invent. Math. \textbf{124} (1996), 437--449.

\bibitem[NSW13]{NSW}
J.~Neukirch, A.~Schmidt, and K.~Wingberg, \emph{Cohomology of number fields},
  second ed., Springer, 2013.

\bibitem[Oes]{Oesterle_}
J.~Oester\'e, \emph{Torsion des courbes elliptiques sur les corps de nombres},
  preprint.

\bibitem[PRV12]{PaladinoRV_12}
L.~Paladino, G.~Ranieri, and E.~Viada, \emph{{On local-global divisibility by
  $p^n$ in elliptic curves}}, Bull. Lon. Math. Soc. \textbf{44} (2012), no.~5,
  789--802.

\bibitem[PRV14]{PaladinoRV_14}
Laura Paladino, Gabriele Ranieri, and Evelina Viada, \emph{On the minimal set
  for counterexamples to the local-global principle}, J. Algebra \textbf{415}
  (2014), 290--304. \MR{3229518}

\bibitem[Ser68]{Serre_abelian}
J.-P. Serre, \emph{Abelian $\ell$-adic representations and elliptic curves},
  New York, Benjamin Publ., 1968.

\end{thebibliography}
\bibliographystyle{amsalpha}

\end{document}